\documentclass[11pt]{amsart}
\usepackage{amsmath, amssymb, amscd, mathrsfs, slashed, enumerate, url}
\usepackage{xcolor}
\usepackage[all,cmtip]{xy}
\usepackage{multicol}
\usepackage{indentfirst}
\usepackage{latexsym}
\usepackage{bm}
\usepackage{graphicx}
\usepackage{subfigure,esint}
\usepackage{float}
\usepackage{verbatim}
\usepackage{comment}
\usepackage[top=1in, bottom=1in, left=1.25in, right=1.25in]{geometry}
\usepackage{epsfig,dsfont,amsthm,amsfonts,amsbsy}
\usepackage[colorlinks]{hyperref}
\newcommand{\gpp}{\mathfrak{g}_P}
\newcommand{\gee}{\mathfrak{g}_E}

\newcommand{\MGC}{\mathcal{G}_{\mathbb{C}}}

\newcommand{\lan}{\langle }
\newcommand{\ran}{\rangle}

\newcommand{\MD}{\mathcal{D}}

\newcommand{\MM}{\mathcal{M}}

\newcommand{\MO}{\mathcal{O}}

\newcommand{\End}{\mathrm{End}}

\newtheorem{theorem}{Theorem}[section]

\newtheorem{corollary}[theorem]{Corollary}

\newtheorem{definition}[theorem]{Definition}

\newtheorem{lemma}[theorem]{Lemma}

\newtheorem{proposition}[theorem]{Proposition}
\newtheorem*{remark}{Remark}

\newcommand{\MC}{\mathcal{C}}
\newcommand{\Tr}{\mathrm{Tr}}

\newcommand{\st}{\star}
\newcommand{\we}{\wedge}
\newcommand{\pa}{\partial}

\newcommand{\RP}{\mathbb R^+}
\newcommand{\EBE}{extended Bogomolny equations\;}
\newcommand{\ti}{\times}

\newcommand{\Si}{\Sigma}
\newcommand{\bz}{\bar{z}}
\newcommand{\vp}{\varphi}

\newcommand{\ME}{\mathcal{E}}
\newcommand{\da}{\dagger}
\newcommand{\al}{\alpha}

\newcommand{\na}{\nabla}
\newcommand{\ep}{\epsilon}
\newcommand{\hA}{\widehat{A}}

\newcommand{\hP}{\widehat{\Phi}}

\newcommand{\calC}{\mathcal C}

\newcommand{\RR}{\mathbb R}
\newcommand{\CC}{\mathbb C}

\newcommand{\NP}{\mathrm{NP}}

\newcommand{\mft}{\mathfrak{t}}

\newcommand{\MG}{\mathcal{G}}

\newcommand{\md}{\mathrm{mod}}

\newcommand{\SL}{\mathrm{SL}}

\newcommand{\NPK}{\mathrm{NPK}}
\newcommand{\mfd}{\mathfrak{d}}

\newcommand{\Hit}{\mathrm{Hit}}

\newcommand{\mfg}{\mathfrak{g}}

\newcommand{\mbC}{\mathbb{C}}
\newcommand{\Me}{{M_{\ep}}}
\newcommand{\Mod}{{\mathrm{mod}}}
\newcommand{\nk}{\mathbf{k}}
\newcommand{\mEBE}{\mathrm{EBE}}
\newcommand{\Pinf}{P_{\infty}}
\newcommand{\PNP}{P_{\infty}^{\NP}}
\newcommand{\PNPK}{P_{\infty}^{\NPK}}
\newcommand{\mKW}{\mathrm{KW}}
\newcommand{\SUn}{\mathrm{SU(n)}}

\usepackage[utf8]{inputenc}
\usepackage[english]{babel}
\usepackage{fancyhdr}

\begin{document}
\title[Classification of Nahm pole solutions]{Classification of Nahm
  pole solutions of the Kapustin-Witten equations on $S^1\ti\Si\ti\RP$}
\author{Siqi He} 
\address{Simons Center for Geometry and Physics, StonyBrook University\\Stonybrook, NY 11794}
\email{she@scgp.stonybrook.edu}
\author{Rafe Mazzeo}
\address{Department of Mathematics, Stanford University\\Stanford,CA 94305}
\email{rmazzeo@stanford.edu}

\begin{abstract}
In this note, we classify all solutions to the $\SUn$ Kapustin-Witten equations on $S^1\ti\Si \ti \RP$,
where $\Si$ is a compact Riemann surface, with Nahm pole singularity at $S^1\ti\Si \ti \{0\}$. 
We provide a similar classification of solutions with generalized Nahm pole singularities along a simple divisor (a ``knot'')
in $S^1\ti\Si \ti \{0\}$.
\end{abstract}

\maketitle

\section{Introduction}
An important conjecture by Witten \cite{witten2011fivebranes} posits a relationship between the Jones polynomial of a knot
and a count of solutions to the Kapustin-Witten equations.  More specifically, let $K$ be a knot in $X = \RR^3$ or $S^3$,
and fix an $\SUn$ bundle $P$ over $X \ti \RP$ with associated adjoint bundle $\gpp$.  The Kapustin-Witten (KW) equations
\cite{KapustinWitten2006} are equations for a pair $(A, \Phi)$, where $A$ is a connection on $P$ and $\Phi$ is a $\gpp$-valued
$1$-form. We augment these with the singular Nahm pole boundary conditions at $y=0$ (where $y$ is a linear variable on
the $\RP$ factor), and with an additional singularity imposed along $K \ti \{0\}$ . The conjecture states that an appropriate count of
solutions to the KW equations with these boundary conditions computes the Jones polynomial. One can define these equations
when $X$ is a more general Riemannian $3$-manifold, and in that case this gauge-theoretic enumeration 
may lead to new $3$-manifold invariants when $K = \emptyset$, or to a generalization of the Jones polynomial
for $K$ lying in a general $3$-manifold, see \cite{Witten2014LecturesJonesPolynomial, gaiotto2012knot}.

The core of all of this is to investigate the properties of the moduli space of solutions.  Significant partial
progress has been made,  see \cite{MazzeoWitten2013,MazzeoWitten2017,He2017,RyosukeEnergy,taubes1982self},
as well as Taubes' recent advance \cite{Taubescompactness} regarding compactness properties.

As usual in gauge theory, it is reasonable to seek to understand a dimensionally reduced version of this problem.
Thus suppose that $X = S^1\ti \Si$, where $\Si$ is a compact Riemann surface of genus $g$.  Solutions which
are invariant in the $S^1$ direction are solutions of the so-called extended Bogomolny equations.  General existence
theorems for solutions of these dimensionally reduced equations were proved in \cite{HeMazzeo2017,MazzeoHe18}.
In the present paper, we adapt arguments from \cite{MazzeoWitten2017} and prove that every solution to
the KW equation on $S^1\ti \Si \ti \RP$ satisfying Nahm pole boundary conditions is necessarily invariant in
the $S^1$ direction.  This leads to a complete classification of solutions in this special case.

\begin{theorem}
Consider the Kapustin-Witten equations on $S^1\ti \Si \ti \RP_y$ for fields satisfying the Nahm pole boundary
condition at $y=0$ (with no knot singularity) and which converge to a flat $\SL(n,\CC)$ connection as $y\to\infty$.  
\begin{itemize}
\item [i)]  There are no solutions if $g=0$;
\item [ii)] There is a unique solution (up to unitary gauge equivalence) if $g=1$;
\item[iii)] If $g>1$, there exists a solution if and only if the limiting flat connection as $y \to \infty$
  lies in the Hitchin section in the $\SL(n,\CC)$
 Hitchin moduli space, and in that case, this solution is unique up to unitary gauge.
\end{itemize}
\end{theorem}
Part ii) here largely comes from the uniqueness theorem in \cite{MazzeoWitten2013} for solutions
on $\RR^3 \ti \RP$. The Hitchin section in part iii) is also known as the Hitchin component of
the $\mbox{SL(n, $\mathbb R$)}$ representation variety, cf.\ \cite{hitchin1992lie}. We recall that there
are in fact $n^{2g}$ equivalent Hitchin components, depending on the different choices of spin structure.

Next suppose that the knot $K \subset S^1 \ti \Si$ is a union of `parallel' copies of $S^1$, $K = \sqcup_i(S^1 \ti \{p_i\})$.
The Nahm boundary conditions at a knot require that we specify a weight $\nk^i$, i.e., an $(n-1)$-tuple of positive integers
$(k_1^i, \ldots, k_{n-1}^i) \in \mathbb N^{n-1}$ for each component $K_i$. 
\begin{theorem}
Consider the Kapustin-Witten equations on $S^1 \ti \Si \ti \RP_y$ for fields which satisfy the Nahm pole boundary condition
with knot singularities with weights $\nk^i$, as described above, along $K\ti\{0\}$, where $K = \sqcup_i (S^1 \ti \{p_i\})$, 
and which converge to a flat $\SL(n,\CC)$ connection, corresponding to a stable Higgs pair $(\ME,\vp)$, as $y \to \infty$. 
\begin{itemize}
\item [i)] There are no solutions when $g=0$; 
\item [ii)] If $g>1$ and $\rho$ is irreducible, there exists a solutions with these boundary conditions at $K$ if and only if there exists a holomorphic line subbundle $L$ of
  $\ME$ such that the data set $\mfd (\ME,\vp,L)=\{(p_i,\nk_i)\}$,
\end{itemize}
\end{theorem}
\noindent The definition of data sets $\mfd (\ME,\vp,L)$ is recalled in Section \ref{datasetdefinition}.

\begin{remark}
We do not discuss the case $g=1$ here. Indeed, it is not clear what the correct existence theory 
for solutions with knot singularities should be in this case.
  \end{remark}

\begin{corollary}
There exists, up to unitary gauge, at most $n^{2g}$ solutions to the KW equations which converge to the given flat
connection associated to $(\ME, \vp)$ and with Nahm singularity along $K = \sqcup S^1 \ti \{p_i\}$. 
\end{corollary}
The knot points $p_i$ and the weights $\nk^i$ determine the divisor $D=\sum_ip_i\sum_jk_j^i$. 
\begin{theorem}
  If $\deg D$ is not divisible by $n$, there exist no Nahm pole solutions to the KW equations with knot singularity along $K$.
  In particular, there are no solutions to the $\mathrm{SU}(2)$ \EBE with only a single knot singularity of weight $1$. 
\end{theorem}

\textbf{Acknowledgements.} The first author would like to thank Simon Donaldson for numerous helpful discussions.
The second author was supported by the NSF grant DMS-1608223.

\section{The Kapustin-Witten Equations and the Nahm Pole Boundary Conditions}
We begin with some background materials on the Kapustin-Witten equations \cite{KapustinWitten2006} and Nahm pole boundary conditions:
\subsection{The Kapustin-Witten Equations}
Let $(M,g)$ be a Riemannian $4$-manifold, and $P$ an $\SUn$ bundle over $M$ with the adjoint bundle $\gpp$. The Kapustin-Witten equations 
for a connection $A$ and a $\gpp$-valued $1$-form $\Phi$ are
\begin{equation}
\begin{split}
F_A-\Phi\we\Phi+\st d_A\Phi=0,\ \ d_A^{\st}\Phi=0.
\label{KW}
\end{split}
\end{equation}
When $M$ is closed, all solutions to the KW equations are flat $\SL(n,\CC)$ connections \cite{KapustinWitten2006}. Indeed, in this setting, a Weitzenb\"ock formula shows that solutions must satisfy the decoupled equations
\begin{equation}
  F_A-\Phi\we\Phi=0,\;d_A\Phi=0,\;d_A\st\Phi=0,
  \label{flatconnection}
\end{equation}
or equivalently, $F_{\mathcal A} = 0$ where $\mathcal A := A + i\Phi$ and $d_A \st \Phi = 0$. 

Following \cite{witten2011fivebranes,Witten2014LecturesJonesPolynomial}, the main case of interest here is when $M=X\ti\RP$,
where $X$ is a closed 3-manifold and $\RP:=(0,\infty)$ with linear coordinate $y$. From now on, we fix a Riemannian metric
on $X$ with volume $1$, and endow $X\ti\RP$ with the product metric.

\subsubsection{The Nahm Pole Boundary Condition}
Let $G:=\SUn$, with Lie algebra $\mfg$ and choose a principal embedding $\varrho:\mathfrak{su}(2)\to\mfg$ as well as a
global orthonormal coframe $\{\mathfrak{e}^*_a,\ a=1,2,3\}$ of $T^*X$, which is possible since $X$ is parallelizable.
Next, choose a section $e$ of $T^*X \otimes \gpp$, $e = \sum \mathfrak t_a e_a^*$ for some everywhere nonvanishing
sections $\mathfrak{t}_a$, $a=1,2,3$, of the adjoint bundle $\gpp$ which satisfy the commutation relations
$[\mathfrak{t}_a,\mathfrak{t}_b]=\epsilon_{abc}\mathfrak{t}_c$, and which lie in the conjugacy class of the 
image of $\varrho$. This choice of $e$ is called a {\it dreibein form}.

\begin{definition}
With all notation as above, the pair $(A,\Phi)$ satisfies the \textbf{Nahm pole boundary condition} at $y=0$ if, in some gauge,
$A=A_0 + \MO(y^{\ep})$ and $\Phi=\frac{e}{y}+\MO(y^{-1+\ep})$ for some $\ep>0$.
\end{definition}

The rationale for this name is that the dimensional reduction of the KW equations to $\RP$ are the Nahm equations, and in
this case $(0,\frac{e}{y})$ is a `standard' solution of the Nahm equations with a so-called pole at $y=0$.  We remark
also that as proved in \cite{MazzeoWitten2013}, it is sufficient to assume that $A = \MO(y^{-1+\ep})$, since the regularity
theory for solutions shows that there is automatically a leading coefficient $A_0$. 
\subsubsection{The Nahm Pole Boundary Condition with Knot Singularities}
A generalization of this boundary condition incorporates certain `knot' singularities at $y=0$. Before describing this,
recall from \cite{witten2011fivebranes} the model solution when $G = \mathrm{SU(2)}$ and $X=\RR^3 =
\mathbb{R}\ti\mathbb{C}$ with coordinate $(x_1,z = x_2 + ix_3)$. Introduce spherical coordinate $(R,s,\theta)$ in the $(z,y)$
half-space: $z=re^{i\theta}$, $R=\sqrt{r^2+y^2}$, $y=R\sin s$, $r= |z| = R\cos s$. The model knot is the line $(x_1,0,0)\subset
\mathbb{R}^3 \ti \{0\}$. Writing $\Phi=\phi_zdz+\phi_{\bz} d\bz+\phi_1 dx_1+\phi_ydy$, the model solution of weight $k$ takes the form
\begin{equation}
\begin{split}
A&=-(k+1)\cos^2 s\frac{(1+\sin s)^k-(1-\sin s)^k}{(1+\sin s)^{k+1}-(1-\sin s)^{k+1}}d\theta \left(\begin{array}{cc}
\frac i2& 0\\
0& \frac i2
\end{array}\right),\\
\phi_z&=\frac{2(k+1)e^{ik\theta}\cos^k s}{R(1+\sin s)^{k+1}-R(1-\sin s)^{k+1}}\left(\begin{array}{cc}
0& 1\\
0& 0
\end{array}\right),\\
\phi_1&=\frac{k+1}{R}\frac{(1+\sin s)^{k+1}+(1-\sin s)^{k+1}}{(1+\sin s)^{k+1}-(1-\sin s)^{k+1}}\left(\begin{array}{cc}
\frac i2& 0\\
0& \frac i2
\end{array}\right),\;\phi_y=0.
\end{split}
\end{equation}
There is a less explicit model solution when $G= \SUn$, due to Mikhaylov \cite{Mikhaylov2012solutions}. The weight in that case
is an $(n-1)$-tuple $\nk=(k_1,\cdots,k_{n-1})$, and the corresponding solution is denoted $(A^{\md}_{\nk},\Phi^{\md}_{\nk})$.
As in the case $n=2$, $|A^{\md}_{\nk}|\sim R^{-1}s^0$ and $|\Phi^{\md}_\nk|\sim R^{-1}s^{-1}$ near $z=0, y=0$. 

In general, given a knot $K\subset X\ti\{0\}$, introduce local coordinates $(x_1,z = x_2 + i x_3 ,y)$
near $K$, where $K = \{z = y = 0\}$ and $t$ is a coordinate along $K$. We can use 
cylindrical coordinates $(R,s,\theta, x_1)$ near $K$, where $y=R\sin s$, $z=R\cos s e^{i\theta}$. Then, as in \cite{witten2011fivebranes,MazzeoWitten2017}, we make the 
\begin{definition}
With $P$ and $G$ as above, and $K \subset X$ a knot, then $(A,\Phi)$ satisfies \textbf{Nahm pole boundary condition with knot $K$ and
weight $\nk$} if in some gauge 
	\begin{itemize}
		\item [i)] $(A,\Phi)$ satisfies the Nahm pole boundary condition.away from knots $K$, 
		\item [ii)] near $K$, $A=A^{\md}_{\nk}+\MO(R^{-1+\ep}s^{-1+\ep})$, $\Phi=\Phi^{\md}_{\nk}+\MO(R^{-1+\ep}s^{-1+\ep}).$
	\end{itemize}
\end{definition}

\subsubsection{The Boundary Condition at $y=\infty$}
We must also impose an asymptotic boundary condition at the cylindrical end, as $y \to \infty$. We change to a temporal gauge,
i.e., so that $A_y  \equiv 0$. Then writing $\Phi=\phi+\phi_ydy$ (so $\phi$ includes the $\phi_1$ part), the KW equations
become flow equations
\begin{equation}
\begin{split}
&\pa_y A=\st d_A\phi+[\phi_y,\phi],\\
&\pa_y\phi=d_A\phi_y+\st(F_A-\phi\we\phi),\\
&\pa_y\phi_y=d_A^{\st}\phi.
\label{flowequation}
\end{split}
\end{equation}

We shall assume that $(A,\Phi)$ converges to a "steady-state" ($y$-independent) solution as $y \to \infty$, which
is then necessarily a flat $SL(n,\CC)$ connection. The $y$-independence, together with the equations \eqref{flatconnection}
yield that $[\phi,\phi_y]=d_A\phi_y=0$; this shows that if $\phi_y \neq 0$, then $\mathcal A$ is reducible.
\begin{proposition}
  If $(A,\Phi)$ satisfies the KW equations together with Nahm pole boundary conditions (possibly with knots), and converges to an irreducible
  flat $\SL(n,\CC)$ connection as $y \to \infty$, then $\phi_y \equiv 0$.
  \label{vanphiy}
\end{proposition}
Indeed, the hypothesis and the remark above shows that $\lim_{y\rightarrow +\infty}\phi_y=0$.  A well-known vanishing theorem
then implies that $\phi_y \equiv 0$, see \cite[Page 36]{taubes2013compactness} or \cite[Corollary 4.7]{He2017} for a proof.
We assume henceforth, as in \cite{Taubescompactness}, that $\phi_y
\equiv 0$.

We now define the moduli spaces
\begin{equation}
\begin{split}
\MM_{\mathrm{NP}}^{\mKW}:=\{(A,\Phi): \ \mKW(A,\Phi)=0, \  (A,\Phi) \mbox{ converges to a flat }
\SL(n,\CC)\;connection\\ \mbox{as} \ y\to\infty\ \mbox{and}\ \mbox{ satisfies the Nahm Pole boundary condition at} \ y=0\}/\MG_0,
\end{split}
\end{equation}
and
\begin{equation}
\begin{split}
  \MM^{\mKW}_{\mathrm{NPK}}  :=\left\{ (A,\Phi): \mKW(A,\Phi)=0,\ (A,\phi,\phi_1)\ \mbox{ satisfies the Nahm pole} \right. \\
\mbox{ boundary condition with knot $K$ and converges to a flat } \\SL(n,\CC) \ \textrm{connection as} \ y\to\infty \}/\mathcal{G}_0,
\label{complexgeometrymodulispace}
\end{split}
\end{equation}
where $\MG_0$ is the space of gauge transformations preserving the boundary conditions.

\subsection{The Regularity theorems of Nahm pole Solutions}

We next recall the regularity theory for this singular boundary condition at $y=0$, as developed in \cite{MazzeoWitten2013,MazzeoWitten2017}.
Still working on $X\ti\RP$, fix a smooth background connection $\nabla$, and write $\na_x, \na_y$ for the covariant derivatives in the $x \in X$ and
$y$ directions. 
\begin{theorem}{\cite{MazzeoWitten2013, MazzeoWitten2017}}
\label{expansions} 
Let $(A,\Phi)$ satisfy the KW equations with Nahm pole boundary condition, and write $A = A_0 + a$, $\Phi=\frac{e}{y}+b$ near $y=0$ where $a =
\MO(y^\ep)$, $b = \MO(y^{-1 + \ep})$. Then $a$ and $b$ are polyhomogeneous. 
Furthermore, the leading term 
$A_0 $ of $A$ must correspond, under the intertwining provided by the dreibein $e$, with the Levi-Civita connection on $X$.


If $(A,\Phi)$ satisfies the Nahm pole boundary condition with a knot singularity along $K$ of weight $\nk$ at $y=0$,
then writing
$A=A^{\md}_\nk+a,\;\Phi=\Phi^{\md}_{\nk}+b$, where $(A^{\md}_\nk,\Phi^{\md}_\nk)$ is the model solution and $a, b = \MO(R^{-1 + \ep} s^{-1+\ep})$,
then $a, b$ are polyhomogeneous, i.e., have expansions in positive powers of $R$ and $s$, and nonnegative integer
powers of $\log R$ and $\log s$, with coefficients smooth in the tangential variables. These expansions are of product type at the corner
$R = s= 0$. 
\end{theorem}

\begin{remark}
We recall that a function (or section of some bundle) $u$ is polyhomogeneous on $X \ti \RP$ at $X \ti \{0\}$ if,
near any boundary point, 
\[
u \sim \sum_j \sum_{\ell = 0}^{N_j}  u_{j\ell}(x) y^{\gamma_j} (\log y)^\ell\ \ \mbox{as}\ y \to 0.
\]
Here $x$ is a local coordinate on $X$ and each coefficient $u_{j\ell}(x)$ is $\calC^\infty$, while $\gamma_j$ is a sequence of complex numbers
with real parts tending to infinity. In our setting, the $\gamma_j$ are explicit real numbers calculated in \cite{MazzeoWitten2013}.

The second polyhomogeneity statement, near $K$, may be phrased similarly once we introduce the blowup $[X \ti \RP; K \ti \{0\}]$. This
is a new manifold with corners of codimension two obtained by replacing the knot $K$ at $y=0$ with its inward-pointing spherical normal
bundle. The cylindrical coordinates $(x_1, R,  s, \theta)$ are nonsingular on this space, and the two boundaries are defined by $\{R=0\}$
and $\{s=0\}$. A function or section $u$ is polyhomogeneous on this space if it admits a classical expansion as described above near
each point in the interior of the codimension one boundaries, while near the corner $\{R = s = 0\}$ it admits a product type expansion
\[
u \sim \sum_{j, k}  \sum_{\ell=0}^{N_j} \sum_{m = 0}^{M_j}  u_{j k \ell m}(x_1, \theta) s^{\gamma_j} R^{\mu_k} (\log s)^\ell (\log R)^m,
\]
where as before, each coefficient function is smooth in the variables $t,\theta$ along the corner. In our setting the $\gamma_j$ and $N_j$
are the same numbers as in the previous expansion, while the $\mu_k$ are real numbers calculated (somewhat less explicitly, i.e., only
in terms of spectral data of some auxiliary operator) in \cite{MazzeoWitten2017}.

The paper \cite{Heexpansion18} considers various refined aspects of the higher terms in the expansion in $y$. 

We have described this precise regularity for the sake of completeness, but in fact, we do not use the full power of these expansions here, but only the estimates
\begin{equation*}
  \begin{split}
& |\na_x^{\ell}\na_y^m a|_{\MC^0}\leq C_{\ell,m}y^{2-m+\ep},\ \ |\na_x^{\ell}\na_y^m b|_{\MC^0}\leq C_{l,m}y^{1-m+\ep}, \\
& |\na_{x_1}^{\ell}\na_R^{m}\na_s^n a|_{\MC^0}\leq C_{\ell,m,n}R^{-\ep-m}s^{2-\ep-n},\ \ |\na_{x_1}^{\ell}\na_R^{m}\na_s^n b|_{\MC^0}
\leq C_{\ell,m,n}R^{-\ep-m}s^{1-\ep-n}
\end{split}
\end{equation*}
for any $\ep>0$ and any $\ell,m,n \in \mathbb N$.
\end{remark}


\section{The Extended Bogomolny Equations}
We next recall the dimensional reduction of the Kapustin-Witten equations from $S^1\ti\Si\ti \RP$ to $\Si \ti \RP$, obtained by
considering fields invariant in the $S^1$ direction. This was previously studied in \cite{HeMazzeo2017,MazzeoHe18}, and is
closely related to the Atiyah-Floer approach to counting Kapustin-Witten solutions \cite{gaiotto2012knot}.

Assume on the one hand that the bundle $P$ on $S^1 \ti \Si \ti \RP$ is pulled back from $\Si \ti \RP$.  Changing
notation slightly, given a solution
$(\hA,\hP)$ of the KW equations on $S^1 \ti \Si \ti \RP$, choose a gauge for which the $S^1$ component of $A$ vanishes
and $A_y \equiv 0$ as well. By virtue of the Nahm pole boundary conditions at $y=0$ and the asymptotic condition as
$y \to \infty$, Proposition \ref{vanphiy} gives that $\phi_y = 0$, but we cannot gauge away the $S^1$ component $\phi_1$.
Thus the remaining fields are $(\hA_\Si, \hP_1, \hP_\Si)$. We regard $\hA_\Si$ as a connection $A$ on $\Si$, and
write $\hP_1 = \phi_1$, $\hP_\Si = \phi$. These remaining fields satisfy the {\bf extended Bogomolny equations}
\begin{equation}
\label{Eq_EBE}
\begin{split}
&F_A-\phi\we\phi-\st d_A\phi_1=0\\
&d_A\phi +\st [\phi,\phi_1]= 0,\\
&d_A^{\st}\phi= 0.
\end{split}
\end{equation}

On the other hand, given a solution $(A, \phi, \phi_1)$ of the \EBE on $\Si \ti \RP$, then denoting by $\pi:S^1\ti \Si\ti \RP\to \Si\ti \RP$
the natural projection, we define the connection $\hA = \pi^* A$ and Higgs field $\hP = \pi^{\st}\phi+\pi^{\st}\phi_1dx_1$.
It is straightforward to check that $(\hA,\hP)$ satisfies the KW equations.

Let $D=\{(p_i,\nk_i=(k_1^i,\cdots,k_{n-1}^i))\}$ where for each $i$, $k_j^i$ are non-negative integers with
at least one of them nonzero. 
\begin{definition}
Let $(A,\phi,\phi_1)$ be a solution to the \EBE on $\Si \ti \RP$.
\begin{itemize}
\item [i)] The fields $(A,\phi,\phi_1)$ satisfy the \textbf{Nahm pole boundary condition} if the corresponding fields $(\hA,\hP)$
 satisfy the Nahm pole boundary condition on $S^1 \ti \Si \ti \RP$. 
\item [ii)] Similarly, $(A,\phi,\phi_1)$ satisfies the \textbf{Nahm pole boundary condition with knot data $D$} if the
corresponding pull back fields $(\hA,\hP)$ satisfy the Nahm pole boundary condition with knots at $K_i:=S^1\ti \{p_i\}$ with
weight $\nk_i$.
\end{itemize} 
\end{definition} 

The moduli space we shall consider are:
\begin{equation}
\begin{split}
\MM_{\mathrm{NP}}^{\mEBE}:=\{(A,\phi,\phi_1): \ \mathrm{EBE}(A,\phi,\phi_1)=0, \  (A,\phi,\phi_1) \mbox{ converges to a flat }
\mathrm{SL}(n,\mathbb{C}) \\ \mbox{connection as} \ y\to\infty\ \mbox{and}\ \mbox{ satisfies the Nahm Pole boundary condition 
	at} \ y=0\}/\MG_0,
\end{split}
\end{equation}
and
\begin{equation}
\begin{split}
\MM^{\mEBE}_{\mathrm{NPK}} & :=\left\{ (A,\phi,\phi_1): \mathrm{EBE}(A,\phi,\phi_1)=0,\ (A,\phi,\phi_1)\ \textrm{ satisfies the Nahm pole } \right. \\
& \textrm{boundary condition with knot and converges to a flat} \\ & \left. \mathrm{SL}(n,\mathbb{C}) 
\ \textrm{connection as} \ y\to\infty \right\}/\mathcal{G}_0,
\label{complexgeometrymodulispace}
\end{split}
\end{equation}
where $\MG_0$ is the gauge transformations that preserve the boundary condition.

\subsection{Hermitian-Yang-Mills Structure}
In \cite{gaiotto2012knot,witten2011fivebranes}, it is observed that the extended Bogomolny equations have a Hermitian-Yang-Mills
structure. By this we mean the following.   Let $E$ be complex vector bundle of rank $n$ over $\Si\ti\RP$ with $\det E=0$.
A choice of Hermitian metric $H$ on $E$ induces an $\mathrm{SU}(n)$ structure on this bundle, and we denote by
$\gee$ the associated adjoint bundle.  Writing
\[
d_A=\na_2dx_2+\na_3dx_3+\na_ydy,\ \mbox{and}\ \ \phi=\phi_2dx_2+\phi_3dx_3 =\frac{1}{2}(\vp_z dz+\vp_{\bz} d\bz),
\]
we define the operators 
\begin{equation}
\begin{split}
&\MD_1=(\na_2+i\na_3)d{\bz}=(2\pa_{\bz}+A_1+iA_2)d{\bz},\\
&\MD_2=\operatorname{ad}\vp=[\vp,\cdot ]=[(\phi_2 - i \phi_3) \, dz ,\cdot ], \\
&\MD_3=\na_y-i\phi_1=\pa_y+A_y-i\phi_1. 
\end{split}
\end{equation}
Their adjoints with respect to $H$ are denoted $\MD_i^{\dag_H}$. The \EBE can then be written in the elegant form
\begin{equation}
\begin{split}
&[\MD_i,\MD_j]=0, \ \ i,j=1,2,3,\\
& \frac{i}{2}\Lambda \left([\MD_1, \MD_1^{\da_H}]+[\MD_2,\MD_2^{\da_H}]\right)+ [\MD_3,\MD_3^{\da_H}]=0,
\label{Eq_EBE_Hermitian}
\end{split}
\end{equation}
where $\Lambda:\Omega^{1,1}\to\Omega^0$ is the inner product with the K\"ahler form (normalized as $(i/2)dz\we d\bz$ 
when the metric on $\Si$ is flat). 

The action $\MD_i\to g^{-1}\MD_ig$ of the gauge group $\MG$ preserves the Hermitian metric; the
complex gauge group is denoted $\MGC$. The smaller system $[\MD_i,\MD_j]=0$ is invariant under $\MGC$,
while the full set of equations \eqref{Eq_EBE_Hermitian} is invariant only under $\MG$.  The final equation is a
real moment map condition. Following Donaldson \cite{donaldson1985anti} and Uhlenbeck-Yau 
\cite{uhlenbeck1986existence},  geometric data from the $\MGC$-invariant equations play an
important role in understanding the moment map equation.

\subsection{Higgs Bundles and Flat Connections}
The appearance of Higgs bundles over $\Si$ in this story is motivated by the fact that the
$y$-independent versions of the equations of \eqref{Eq_EBE}, when in addition $\phi_1 = 0$, are
simply the Hitchin equations. 

Recall that a Higgs bundle over $\Si$ is a pair $(\ME,\vp)$ where $\ME$ is a holomorphic bundle of rank $n$
with $\det \ME=0$ and $\vp\in H^0(\End(\ME)\otimes K)$. A Higgs pair (which is an alternate phrase for Higgs bundles)
$(\ME,\vp)$ is called stable if for any holomorphic subbundle $V$ with $\varphi(V)\subset V\otimes K$, we have $\deg (V)<0$,
and polystable if it is a direct sum of stable Higgs pairs. 

Setting $\MD_3 = 0$ in the \EBE (or alternately, considering only the equations for $\MD_1$ and $\MD_2$ on each slice $\Si_y := \Si \ti \{y\}$),
we obtain the Hitchin equations: 
\begin{equation}
F_H+[\vp,\vp^{\st_H}]=0,\;\bar{\pa}\vp=0.
\label{Hitchinequation}
\end{equation}
The initial term $F_H$ is the curvature of the Chern connection $\nabla_H$ associated to $H$ and the holomorphic structure, and
$\vp^{\st_H}$ is the adjoint with respect to $H$.  Irreducibility of the fields $(A,\vp+\vp^{\st_H})$ is defined in the obvious way.
One may regard \eqref{Hitchinequation} as an equation for the fields $(A, \vp)$ or else for the Hermitian metric $H$;
we consider $H$ as the variable here.
\begin{theorem}{\cite{Hitchin1987Selfdual}}
\label{nonabelianhodge}
For any Higgs pair $(\ME,\vp)$ on $\Si$, there exists an irreducible solution $H$ to the Hitchin equations 
if and only if this pair is stable, and a reducible solution if and only if it is polystable.
\end{theorem}

To any solution $H$ of \eqref{Hitchinequation} we associate the flat $\SL(n,\CC)$ connection $D = \nabla_H + \vp + \vp^{\st_H}$.
This determines, in turn, a representation $\rho: \pi_1(\Si) \to \mathrm{SL}(n,\CC)$ which is well-defined up to conjugation. Irreducibility
of the solution is the same as irreducibility of the representation, while complete reducibility corresponds to the fact that $\rho$ is reductive.
The map from flat connections back to solutions of the Hitchin system is defined as follows: first find a harmonic metric, cf.\, \cite{corlette1988},
which determines a decomposition $D = D^{\mathrm{skew}} + D^{\mathrm{Herm}}$ into skew-Hermitian and Hermitian parts.  After that,
the further decomposition $D^{\mathrm{Herm}} = \vp + \vp^{\star_H}$ determines $\vp$, and hence the Higgs bundle $( (D^{\mathrm{skew}})^{0,1}, \vp)$.

Denoting by $\MM_{\mathrm{Higgs}}:=\{(\ME,\vp)\}^{\mathrm{stable}}/\MGC$ the moduli space of stable $\SL(n,\mathbb{C})$ Higgs bundle, we are
then led to define
\begin{equation}
\label{Pinf}
\PNP:\MM^{\mEBE}_{\NP}\to\MM_{\mathrm{Higgs}},\;\PNPK:\MM^{\mEBE}_{\mathrm{NPK}}\to\MM_{\mathrm{Higgs}};
\end{equation}
this is the map which assigns to a solution $(A,\phi,\phi_1)$ of the \EBE its limiting flat connection, and then, under Theorem \ref{nonabelianhodge}, the corresponding Higgs bundle.

The Hitchin fibration is the map
\begin{equation}
\begin{split}
&\pi:\MM_{\mathrm{Higgs}}\to \oplus_{i=2}^n H^0(\Si,\;K^{i})\\
&\pi(\vp)=(p_2(\vp),\;\cdots,\;p_{n}(\vp)),
\label{HitchinFiberationMap}
\end{split}
\end{equation}
where $\det (\lambda -\vp) = \sum \lambda^{n-j} (-1)^j p_{j}(\vp)$. By \cite{hitchin1987stable}, this is a proper map.

We next introduce the Hitchin component (also called the Hitchin section). Choose a spin structure $K^{\frac{1}{2}}$ and set $B_{i}=i(n-i)$.
Now define the Higgs bundle $(\ME, \vp)$, where 
\begin{equation}
\begin{aligned}
\ME: & =S^{n-1}(K^{-\frac{1}{2}}\oplus K^{\frac{1}{2}})=K^{-\frac{n-1}{2}}\oplus K^{-\frac{n-1}{2}+1}\oplus\cdots\oplus K^{\frac{n-1}{2}} \\[5mm]
\vp & =\begin{pmatrix}
0 & \sqrt{B_1} & 0 &\cdots& 0\\
0 & 0 & \sqrt{B_2} & \cdots& 0\\
\vdots & \vdots &\ddots & &\vdots\\
0 & \vdots & &\ddots &\sqrt{B_{n-1}}\\
q_{n}& q_{n-1} & \cdots & q_2 &0
\end{pmatrix}.
\end{aligned}
\label{HitchincomponentHiggsfield}
\end{equation}
The constant $\sqrt{B_i}$ in the $(i,i+1)$ entry represents this multiple of the natural isomorphism $K^{-\frac{n-1}{2} + i  } \to 
K^{-\frac{n-1}{2} + i -1 }\otimes K$, and similarly, $H^0(\Si,K^{n-i}) \ni q_{n-i} :K^{- \frac{n-1}{2} + i }\to K^{ \frac{n-1}{2}} \otimes K$. 
The Hitchin component $\MM_{\mathrm{Hit}}$ is the complex gauge orbit of this family of Higgs bundle,
\begin{equation}
\MM_{\mathrm{Hit}}:=\{(\ME:=S^{n-1}(K^{-\frac{1}{2}}\oplus K^{\frac{1}{2}}),\ \vp\ \mbox{as in}\ \eqref{HitchincomponentHiggsfield})\}/\MGC.
\end{equation}
The following theorem explains its importance.
\begin{theorem}{\cite{hitchin1992lie}}
Every element in $\MM_{\mathrm{Hit}}$ is a stable Higgs pair. 
Furthermore, the map assigning to each element of $\oplus_{i=2}^{n}H^0(\Si,K^{i})$ the unique solution of the Hitchin
equations corresponding to the associated Higgs pair is a diffeomorphism to one of the $n^{2g}$
choices for the Hitchin component; thus its inverse, the restriction of the Hitchin fibration
$\pi|_{\MM_{\mathrm{Hit}}}$, is also a diffeomorphism.
\end{theorem}
Note that the image of this map is only one component of the space of all irreducible flat $\SL(n,\RR)$ connections,
which explains the name `Hitchin component.'

\subsection{The Kobayashi-Hitchin Correspondence}
\label{datasetdefinition}
We now recall the Kobayashi-Hitchin correspondence for the \EBE moduli space \cite{gaiotto2012knot,HeMazzeo2017,MazzeoHe18}.

As noted earlier, from the Hermitian structure in \eqref{Eq_EBE_Hermitian} and the commutation relationship $[\MD_1,\MD_2]=0$,
we obtain a Higgs bundle $(\ME_y,\vp_y)$ on each slice $\Si\ti\{y\}$. The commutation relationship $[\MD_3,\MD_1]=[\MD_3,\MD_2]=0$
means that parallel transport by $\MD_3$ identifies these Higgs bundles for different values of $y$.

Suppose first that the solution of the \EBE satisfies the Nahm pole boundary condition without knots. As explained in more
detail in \cite[Section 4]{MazzeoHe18},  there is a holomorphic line subbundle $L \subset E$ determined by the property that the
parallel transports (under $\MD_3$ parallel transport) of its sections vanish at the fastest possible rate as $y \to 0$, measured with
respect to the Hermitian metric $H$.  In other words, a solution of the \EBE satisfying these boundary conditions determines
a triple $(\ME,\vp,L)$, consisting of a Higgs bundle and a line subbundle.

More generally, consider any triple $(\ME, \vp, L)$ where $L$ is any holomorphic line subbundle of $\ME$. Define holomorphic maps
$$
f_i:=1\we\vp\cdots \we\vp^{i-1} \in H^0(\Si; L^{- i}\otimes \we^{i}E\otimes K^{\frac{i(i-1)}{2}}),\ \ 1 \leq i \leq n.
$$
Note that 
$Z(f_j)-Z(f_{j-1})=\sum_i k^{i}_j p_i$ for some $k^{i}_j \in \mathbb N$. Setting $\nk_i:=(k^i_1,\cdots, k^i_{n-1}) \in \mathbb N^{n-1}$,
then we define the knot data set to be $\mathfrak d(\ME,\vp,L):=\{(p_i,\nk_i)\}$.  Note the important special case
(which holds by noting that $f_n\neq 0$ everywhere): 
\begin{proposition}{\cite[Section 4]{MazzeoHe18}}
If $\mathfrak d(\ME,\vp,L)=\emptyset$, then $(\ME,\vp)\in\MM_{\Hit}$ and $L\cong K^{\frac{n-1}{2}}$.
\end{proposition}

We then state the main equivalences between the \EBE moduli spaces and the spaces of triples $(\ME, \vp, L)$, first for data in
the Hitchin component and then for general data.
\begin{theorem}{\cite{HeMazzeo2017,MazzeoHe18}}
\label{Thm_KobayashiHitchinNP}
There is a diffeomorphism of moduli spaces
\[
\MM^{\mEBE}_{\NP}\cong \MM_{\mathrm{Hit}}.
\]
More specifically, recall the map $\Pinf^{\NP}$ from \eqref{Pinf}.
\begin{itemize}
\item [i)]  For any $(\ME,\vp)\in\MM_{\Hit}$, there exists a unique Nahm pole solution $(A,\phi,\phi_1)\in \MM^{\mEBE}_{\NP}$
such that $P_{\infty}^{\NP}(A,\phi,\phi_1)=(\ME,\vp)$; 
\item [ii)]  Given any Higgs bundle $(\ME,\vp)\notin\MM_{\Hit}$, there is no solution to the extended Bogomolny equations which
 converges to the flat connection determined by $(\ME,\vp)$. In other word, $(\Pinf^{\NP})^{-1}(\ME,\vp)=\emptyset$.
\end{itemize}
\end{theorem}

\begin{theorem}{\cite{HeMazzeo2017,MazzeoHe18}}
 \label{Thm_KobayashiHitchinKnot}
 Fix a data set $D=\{(p_i,\nk_i=(k_1^i,\cdots,k_{n-1}^i))\}$.  If $(\ME, \vp)$ is any stable Higgs bundle over $\Si$ with genus $g(\Si)>1$,
 there exists a solution to the \EBE satisfying the general Nahm pole boundary condition with knot singularities at $p_i$ with weight $\nk_i$
 if and only if there exists a line bundle $L\subset \ME$ such that $\mathfrak d(\ME,\vp,L)=D$. In other words, there is a bijection
 \[
   \MM^{\mEBE}_{\NPK}\cong \{(\ME,\vp,L)\}/\MGC,
 \]
where the pairs $(\ME,\vp)$ on the right are stable Higgs bundles and $L \subset \ME$ is a line subbundle.
\end{theorem}
\begin{remark}
  Notice that in the second result, when knot singularities are allowed, we do not claim that this bijection of moduli spaces
  is a diffeomorphism. Indeed, while the space of triples $(\ME, \vp, L)$ maps onto the space of all stable Higgs pairs, i.e., onto
  the entire Hitchin moduli space, it   is not clear that this space of triples is even a manifold. 

  As a second remark, if $(\ME,\vp)$ is polystable, it seems likely that there are no solutions to the \EBE which satisfy Nahm
  pole boundary conditions   with knot singularities which converge to $(\ME,\vp)$. However, we do not prove this.
\end{remark}

\section{A Weitzenb\"ock Identity for the Kapustin-Witten Equations}
In this section, we establish a Weitzenb\"ock identity analogous to the one in \cite{MazzeoWitten2017}, and use this to show
that all solutions to the KW equations over $M:=S^1\ti \Si \ti \RP$ are invariant in the $S^1$ direction, hence determine solutions
to the extended Bogomolny equations.   In all the following, we use coordinates $x_1 \in S^1$, $z \in \Si$ and $y \in \RP$. 
\subsection{Weitzenb\"ock Identity}
As before, let $P$ be an $\mathrm{SU}(n)$ bundle over $M:=S^1\ti  \Si\ti \RP$, and fix a connection $\hA$ and a
$\gpp$-valued $1$-form $\hP$ on $M$; assume that $\hA_1 = \hA_y = \hP_y \equiv 0$.
Write $d_{\hA}=d_A+dx_1\we \na_1$ and $\hP=\phi+\phi_1 dx_1$. We also fix a
product metric on $S^1\ti\Si\ti\RP$ with orientation $dx_1\we dA_\Sigma\we dy$. 

Now write $F_{\hA}=F_A+B_A\we dx_1$; the Bianchi identity $d_{\hA}F_{\hA}=0$ is equivalent to
\begin{equation}
\label{Bianchi}
\na_1 F_A+d_AB_A=0
\end{equation}
In the following, we write $\st_4$ and $\st$ for the Hodge star operators on $M$ and $S^1\ti\Si$, respectively.

We first compute
\begin{equation}
\begin{split}
&F_{\hA}- \hP\we \hP+\st_4 d_{\hA}\hP \\
& \qquad \quad  = (F_A-\phi\we\phi+\st(\na_1\phi-d_A\phi_1))+(B_A-[\phi,\phi_1]-\st d_A\phi)\we dx_1,\\
&d_{\hA}^{\st_4}\hP=d_A^{\st}\phi-\na_1\phi_1.
\end{split}
\end{equation}	
Next, for any $\ep\in(0,1)$, write $M_{\ep}:=S^1\ti\Si\ti [\ep, \ep^{-1}]$. Then 
\begin{equation}
\begin{split}
  \int_{M_{\ep}}|KW|^2 &=  \int_{M_{\ep}}|F_A-\phi\we\phi+\st(\na_1\phi-d_A\phi_1)|^2 \\ & \qquad \ \ 
  + |B_A-[\phi,\phi_1]-\st d_A\phi|^2+|d_A^{\st}\phi-\na_1\phi_1|^2\\
  & =\int_{M_{\ep}}|F_A-\phi\we\phi-\st d_A\phi_1|^2+|\na_1\phi|^2+ \\
  & \qquad \ \ |B_A|^2+|[\phi,\phi_1]+\st d_A\phi|^2+|d_A^{\st}\phi|^2+|\na_1\phi_1|^2+\int_{ M_\ep} \chi,
\end{split}
\end{equation}
where
\begin{equation}
\chi:=2\lan F_A-\phi\we\phi-\st d_A\phi_1,\st \na_1\phi\ran-2\lan B_A,\st d_A\phi+[\phi,\phi_1]\ran-2\lan d_A^{\st}\phi,\na_1\phi_1 \ran.
\end{equation}
The inner product here is $\lan A,B\ran:=- \Tr(A\we \st_4 B)$. 

\begin{lemma}
We have the following identities:
\begin{itemize}
\item [i)]  $\quad \lan F_A,\st \na_1\phi \ran-\lan B_A,\st d_A\phi\ran=\na_1\Tr(F_A\we \phi)\we dx_1+d\Tr(B_A\we\phi)\we dx_1$,
\item [ii)]
 \begin{equation*}
\begin{split}
&\lan \st d_A\phi_1,\st \na_1\phi\ran+\lan B_A,[\phi,\phi_1]\ran+\lan \st d_A\phi_1,\st \na_1\phi\ran\\
=&\na_1\Tr(\phi\we \st d_A\phi_1)\we dx_1-\na_1\Tr(\phi_1\we d_A\st \phi\we dx_1)
\end{split}
\end{equation*}
\item [iii)] $\quad \lan \phi\we\phi, \st \na_1\phi\ran=\frac{1}{3}\na_1\Tr(\phi\we\phi\we\phi\we dx_1)$
\end{itemize}
\end{lemma}
\begin{proof}
For (i), we compute 
\begin{equation*}
\begin{split}
\lan F_A, & \st \na_1\phi \ran -\lan B_A,\st_3 d_A\phi\ran\\
&= \Tr(F_A\we \na_1\phi)\we dx_1-\Tr(B_A\we d_A\phi)\we dx_1\\
&= \na_1\Tr(F_A\we \phi)\we dx_1-\Tr(\na_1F_A\we \phi)\we dx_1 \\
& \qquad \qquad +d\Tr(B_A\we\phi)\we dx_1-\Tr(d_AB_A\we\phi)\we dx_1\\
&= \na_1\Tr(F_A\we \phi)\we dx_1+d\Tr(B_A\we\phi)\we dx_1,
\end{split}
\end{equation*}
where the last step uses \eqref{Bianchi}.
		
Next, for (ii), 
\begin{equation*}
\begin{split}
\ \ \lan \st d_A\phi_1, & \st \na_1\phi \ran \\
& = \na_1\Tr(\phi\we \st d_A\phi_1)\we dx_1-\Tr(\phi\we \na_1(\st d_A\phi_1))\we dx_1\\
& = \na_1\Tr(\phi\we \st d_A\phi_1)\we dx_1-\Tr(\phi\we \st B_A\we \phi_1)\we dx_1,
\end{split}
\end{equation*}
\begin{equation*}
\begin{split}
\ \lan d_A^{\st}\phi, & \na_1\phi_1\ran \\
& = -\Tr(\na_1\phi_1\we d_A\st \phi)\we dx_1\\
& = -\na_1\Tr(\phi_1\we d_A\st \phi\we dx_1)+\Tr(\phi_1\we B_A\we \st \phi)\we dx_1,
\end{split}
\end{equation*}
and 		
\begin{equation*}
\lan B_A, [\phi,\phi_1]\ran=-\Tr(B_A\we \st \phi\phi_1-B_A\we\phi_1\we \st \phi)\we dx_1.
\end{equation*}
Adding these three equalities yields (ii).  The proof of (iii) is straightforward.
\end{proof}

\begin{corollary}
We have
\begin{equation*}
\label{boundaryterms}
\begin{split}
\int_{\Me} \chi=&\int_\Me \na_1\Tr(2F_A\we\phi-\frac{2}{3}\phi^3-2\phi\we \st d_A\phi_1+\phi_1\we d_A\st \phi)\we dx_1\\
&+2\int_\Me d \, \Tr(B_A\we \phi)\we dx_1.
\end{split}
\end{equation*}
\end{corollary}
\begin{lemma}
\label{limitflatSLCconnection}
Let $A^\rho+\phi^\rho$ be a flat $SL(n,\mbC)$ connection over $S^1\ti\Si$, and write $A^\rho=A^\rho_1+A^\rho_\Si$,
$\phi^{\rho}=\phi^{\rho}_1+\phi^{\rho}_\Si$. 
\begin{itemize}
\item [i)] If we write $F_{A^{\rho}}=B_{A^{\rho}}\we dx_1+E_{A^{\rho}}$, then $B_{A^{\rho}}=0$;
\item [ii)] Up to a unitary gauge transformation, we can assume $A^{\rho}_1$ and $\phi^\rho_1$ are invariant in the $\Si$ directions
and $A^{\rho}_\Si$, $\phi^{\rho}_\Si$ are invariant in the $S^1$ directions.
\item [iii)] Up to a unitary gauge transformation, $\phi_1^{\rho}=0$. 
\end{itemize}
\end{lemma}
\begin{proof}
Items i) and ii) follow from the fact that $\pi_1(\Si\ti S^1)=\pi_1(\Si)\ti\pi_1(S^1)$.
  
For iii), observe that $A^{\rho}_1$ and $\phi^{\rho}_1$ come from the contribution of $\pi_1(S^1)\to SL(n,\mbC)$. Since $\pi_1(S^1)$ is abelian,
and $A^{\rho}_1+i\phi^{\rho}_1$ is an unitary connection, we obtain that $\phi^{\rho}_1=0$. 
\end{proof}	

We now prove vanishing of the second part of the boundary contribution:
\begin{lemma}  Suppose that $(A, \phi, \phi_1)$ is a solution to the extended Bogomolny equations. 
\begin{itemize}
\item [i)] If $(A,\phi)$ satisfies the Nahm pole boundary conditions at $y=0$, with or without knot singularities, then
  $$
  \lim_{\ep\to0}\int_{S^1\ti\Si\ti\{\ep\}}\Tr(B_A\we\phi)=0;
 $$
\item [ii)] If $(A,\phi)$ converges to a flat $\SL(n,\CC)$ connection as $y \to \infty$, then
  $$
  \lim_{\ep\to0}\int_{S^1\ti\Si\ti (1/\ep)}  \Tr(B_A\we\phi)=0.
 $$
\end{itemize}
\end{lemma}
\begin{proof}
  First consider i). Away from knots, Theorem \ref{expansions} gives that $A\sim A_{LC}+\MO(y^{2-\ep})$ for any $\ep>0$,
  which implies that $B_A=B_{A_{LC}}+  \MO(y^{2-\ep})+dy \we (B_{A})_y$. (The `LC' subscript denotes Levi-Civita.) 
 The $dy$ component vanishes in the integration so we may disregard it. In addition,  since we are using the
 product metric, $B_{A_{LC}}=0$.  Finally, since $\phi\sim\frac{e}{y}$, we conclude that $B_A\we\phi\sim \MO(y^{1-\ep})$, so there are no
 boundary contributions in this region.
	
 Near a knot $K$, we use spherical coordinates $(R,s,x_1)$ as before, and consider the boundary term as $R\to 0$. By Theorem \ref{expansions},
 $B_A\sim B_{A^{\Mod}}+\MO(1)\sim \MO(1)$ because $B_{A^{\Mod}}$. In addition, $\phi\sim R^{-1}$, so $B_A\we\phi\sim R^{-1}$.
 Since the volume form is $R^2dR ds dx_1$, this boundary contribution vanishes too. 
	
Part ii) follows directly from the previous lemma.
\end{proof}

The other terms in $\chi$ are derivatives with respect to $x_1$, and hence vanish once we integrate over $S^1$. 

\begin{corollary}
\label{boundarytermvanishes}
Under the previous assumptions, $\int_M\chi=0$.
\end{corollary}

\subsection{$S^1$-invariance}
In summary, we may now conclude the 
\begin{theorem}
\label{S1invariant}
Any solution to the KW equations over $S^1\ti \Si\ti \RP$ satisfying Nahm pole boundary condition at $y=0$ (possibly with knot singularities
at $K=S^1\ti D\ti \{0\}$), and which converges to a flat $\SL(n,\CC)$ connection as $y \to \infty$, is $S^1$ invariant and reduces to
a solution of the extended Bogomolny equations. In addition, $A_1\equiv 0$. 
\end{theorem}
\begin{proof}
 By Corollary \ref{boundarytermvanishes}, any solution to the KW equations with these boundary and asymptotic conditions must satisfy
 $$
 \mathrm{EBE}(A,\phi,\phi_1)=0,\ \na_1\phi=0,\ B_A=0,\ \na_1\phi_1=0,
 $$
where $\mathrm{EBE}$ is the extended Bogomolny equation operator.
		
By Lemma \ref{limitflatSLCconnection}, up to gauge we can assume that $(A,\phi,\phi_1)$ converges to $(A^{\rho},\phi^{\rho},0)$ as
$y\to\infty$, where $A^{\rho},\phi^{\rho}$ is $S^1$ invariant.   Since $(A,\phi,\phi_1)$ is a solution to the \EBE,
Theorem \ref{Thm_KobayashiHitchinNP} and Theorem \ref{Thm_KobayashiHitchinKnot} imply that $(A,\phi,\phi_1)$ is $S^1$ invariant.
From $B_A= \na_1\phi=0$, we obtain $d_AA_1=0$ and $[\phi,A_1]=0$. Irreducibility of solutions to the \EBE with Nahm pole boundary conditions
give finally that $A_1=0$.
\end{proof}

The projection map  $\pi:S^1\ti\Si\ti\RP\to\Si\ti\RP$ naturally induces morphisms
\begin{equation*}
  \begin{split}
    &\pi^{\st}:\MM^{\mEBE}_{\NP}\to\MM^{\mKW}_{\NP}, \\
    & \pi^{\st}:\MM^{\mEBE}_{\NPK}\to\MM^{\mKW}_{\NPK}
  \end{split}
\end{equation*}
We obtain from this the
\begin{corollary}
$\pi^{\st}:\MM^{\mEBE}_{\NP}\to\MM^{\mKW}_{\NP}$ and $\pi^{\st}:\MM^{\mEBE}_{\NPK}\to\MM^{\mKW}_{\NPK}$ are bijections.
\end{corollary}

\section{Classification}
We are now able to complete our main theorem.

\subsection{Case 1: $\Si=S^2$}
\begin{proposition}
There is no Nahm pole solution to the KW equations on $S^1\ti S^2\ti\RP$.
\end{proposition}
\begin{proof}
  By Theorem \ref{S1invariant}, all such solutions must be $S^1$-invariant and reduce to solutions of the Extended Bogomolny
  equations. Hence any such solution would lead to a stable Higgs bundle over $S^2$ with nonvanishing Higgs field.
  However, these do not exist \cite{Hitchin1987Selfdual}.
\end{proof}

\subsection{Case 2: $\Si=T^2$}
We next classify Nahm pole solutions over $T^3\ti\RP$.

Let $M=T^3\ti\RP$ with flat metric $g$.  If $A$ is a connection, then $d_A=\na_A^{\perp}+\na_y$, where $\na^{\perp}_A$ is the covariant
derivative on $T^3$. 

We quote the following identity for solutions of the KW equations from \cite{RyosukeEnergy,MazzeoWitten2013}:
\begin{equation}
\begin{split}
  \int_{M_{\ep}}& |KW(A,\Phi)|^2 \\
  & =\int_{M_{\ep}}(|F_A|^2+|\na_A^{\perp}|^2+| \na_y\phi+\st \phi\we\phi|^2+
  \langle \mathrm{Ric}(\phi), \phi \rangle ) +2\int_{\pa M_{\ep}}\phi\we F_A, 
\label{KWindentities2}
\end{split}
\end{equation}
where $M_\ep:=T^3\ti (\ep,\frac{1}{\ep})$ and $\st$ is the Hodge star operator on $T^3$. 

\begin{proposition}
If $(A,\Phi)$ is a solution to the KW equations over $T^3\ti\RP$ satisfying the Nahm pole boundary conditions, then 
\begin{equation}
\begin{split}
F_A=0,\ ;\na_A^{\perp}\Phi=0,\ \na_y\phi+\st \phi\we\phi=0.
\label{reducedNahmpoletorus}
\end{split}
\end{equation}
\end{proposition}
\begin{proof}
  From \cite{Heexpansion18}, $F_A\sim F_{A_{LC}}+\MO(y^2) = \MO(y^2)$, where $A_{LC}$ is the Levi-Civita connection on $T^3$, 
  but since the metric on $T^3$ is flat, $F_{A_{LC}}=0$.  This shows that $\lim_{\ep\to0}\int_{T^3\ti\{\ep\}}\phi\we F_A=0$.
  Furthermore, since $(A,\Phi)$ converges to a flat connection on $T^3$, Lemma \ref{limitflatSLCconnection} implies that
$\lim_{\ep\to0}\int_{T^3\ti\{\frac{1}{\ep}\}}\phi\we F_A=0$.
\end{proof}

\begin{proposition}
  Let $e$ be a dreibein which is parallel along $T^3$. 
  Then $(0,\frac{e}{y})$ is the only solution to \eqref{reducedNahmpoletorus}.
\end{proposition}
\begin{proof}
  Use the temporal gauge in the $y$-direction, so $\na_y=\pa_y$. Then $\na_y\phi+\st \phi\we\phi=0$ is just the Nahm equations.
  Uniqueness of solutions to the Nahm equations with these boundary conditions implies that $\phi \equiv \frac{e}{y}$ for
  some dreibein $e$.
  Up to a unitary gauge transformation, we can write $e=\sum_{i=1}^3dx_i \mft_i$ where $de=0$, $dx_i$ is an orthogonal basis of
  $T^{\st}T^3$ and the triplet $\mft_i\in\gpp$ satisfies $[\mft_i,\mft_j]=\ep_{ijk}\mft_k$. Finally, $\na_A^{\perp}\Phi=0$ together with
  $de=0$ implies that $A^{\perp}=0$.
\end{proof}

\subsection{Case 3: $g(\Si)>1$}
\begin{proposition}	Let $(A,\Phi)$ be a solution to the KW equations on $S^1\ti\Si\ti\RP_y$ satisfying Nahm pole boundary conditions
and which converges to a flat $\SL(n,\CC)$ connection $(A^\rho, \phi^\rho)$ as $y\to\infty$. If $g(\Si)>1$, then there exists a unique
solution if and only if $\rho$ is $S^1$ independent and lies in the Hitchin component. 
\end{proposition}
\begin{proof}
  By Theorem \ref{S1invariant}, all Nahm pole solutions are $S^1$ invariant and thus satisfy the \EBE and the statement then follows from
  Theorem \ref{Thm_KobayashiHitchinNP}.
\end{proof}

\subsection{Case 4: Knots}
Suppose now that the Nahm pole boundary condition has an additional singularity along the 
knot $K=\cup_iK_i$ where $K_i=S^1\ti \{p_i\}\subset S^1\ti\Si$ with weight $\nk_i=(k_1^i,\cdots, k_{n-1}^i)$.
\begin{theorem}
\label{InvariantKnotstatement}
There is no solution $(A,\Phi)$ to the KW equations over $S^1\ti S^2 \ti\RP_y$ satisfying the Nahm pole boundary conditions
with knots $K_i$ and weight $\nk_i$, and which converges to a flat $\SL(n,\mathbb{C})$ connection as $y \to \infty$.

On the other hand, solutions to these equations with these boundary and asymptotic conditions on $S^1 \ti \Si \ti \RP$ exist
when $g(\Si) > 1$ if and only if there exists a line subbundle $L \subset \ME$, where $(\ME, \vp)$ is the Higgs data
corresponding to the flat bundle at infinity, such that $\mfd (\ME,\vp,L)=\{(p_i,\nk_i)\}$.
\end{theorem}
\begin{proof}
  By Proposition \ref{S1invariant}, solutions in either case are necessarily $S^1$-invariant.  As there
  are no Higgs bundles with non-vanishing Higgs field over $S^2$, there is no solution over $S^1\ti S^2\ti \RP$. The rest of the statement is just Theorem \ref{Thm_KobayashiHitchinKnot}.
\end{proof}

\begin{corollary}
  Let $\rho$ be an irreducible flat $\SL(n,\CC)$ connection. Then there exists at most $n^{2g}$ solutions to the KW
  equations satisfying Nahm pole boundary condition with a knot singularity along $K$ at $y=0$ and which converges
  to $\rho$ in the cylindrical end.
\end{corollary}
\begin{proof}
Denote by $(\ME,\vp)$ the Higgs bundle corresponding to $\rho$. By Theorem \ref{InvariantKnotstatement}, existence
of a solution is equivalent to the existence of a line bundle $L\subset \ME$ for which 
$\mfd(\ME,\vp,L)=\{p_i,\nk_i=(k^i_1,\cdots,k^i_{n-1})\}$. The knot data determines the divisor $D=\sum_{i}p_i(\sum_{j=1}^{n-1}k_j^i)$,
and we have $Z(f_n) = D$ where $f_n:=1\we\vp\we\cdots \we \vp^{n-1}$.
If $L_D$ is the line bundle associated to $D$, then $L^n=L_D^{-1}\otimes K^{\frac{n(n-1)}{2}}$.   However, this determines $L$
only up an $n^{\mathrm{th}}$ root of unity: if $N$ is any line bundle with $N^n=\MO$, then $(L\otimes N)^n=L^n$.
There are $n^{2g}$ choice of $N$, hence $n^{2g}$ possible solutions. However, it is not necessarily the case that 
each $(L \otimes N)^j$ is a subbundle of $\ME$, so there may not be $n^{2g}$ actual solutions.
\end{proof}

\begin{theorem}
  Let $D=\sum_{i}p_i(\sum_{j=1}^{n-1}k_j^i)$ be the divisor determined by the given knot data. If $\deg D$ is not divisible
  by $n$, then there exists no solution. 
\end{theorem}
\begin{proof}
Let $L_D$ be the line bundle associated to $D$ and $(\ME,\vp)$ the Higgs bundle determined by $\rho$. Suppose there exists
a solution; then there exists a subbundle $L \subset \ME$ such that $L^n=L_D^{-1}\otimes K^{\frac{n(n-1)}{2}}$. Therefore,
$\deg D=-n\deg(L)+n(n-1)(g-1)$, so $n$ divides $\deg D$. 
\end{proof}

\begin{corollary}
  Let $K=S^1\ti \{p\}\subset S^1\ti\Si$ with weight $1$ and suppose $g(\Si)>1$. Then there is no $\mathrm{SU}(2)$ solution
  to the KW equations with Nahm pole singularity and knot $K$. 
\end{corollary}

We now focus on the special case where $\rho$ lies in one of the ``non-Hitchin'' components of $\SL(2,\mathbb{R})$
Higgs bundles.  These components are described as follows. Let $\ell$ be a line bundle with $0<\deg \ell <g-1$ and
consider the stable Higgs bundle
$$
\ME=\ell^{-1}\oplus \ell,\ \vp=\begin{pmatrix}
0& \al\\
\beta& 0
\end{pmatrix},
$$
where $\al\in H^{0}(\ell^{-2}\otimes K)$ and $\beta\in H^{0}(\ell^2\otimes K)$ are nontrivial sections. Then the zeroes
of $f_2:=1\we\vp: \ell^2\to K$ coincide with those of $\al$, and the number of zeroes counted with multiplicity equals
$2g-2-2\deg \ell$. 
	
\begin{proposition}
With all notation as above, fix the knot data $D=\sum_i p_i k_i$.
\begin{itemize}
\item [(i)] If $\deg D=2g-2-2\deg \ell$, then there exists a unique Nahm pole solution if and only if $D=\al$ and no
  solution otherwise;
\item [(ii)] if $2g-2>\deg D>2g-2-2\deg \ell$, there is no solution.
\end{itemize}
\end{proposition}
\begin{proof}
With $L_D$ the line bundle for $D$, by Theorem \ref{InvariantKnotstatement} the necessary condition for
existence of a Nahm pole solution is that there exists $L\subset \ME$ such that $L^2=L_D^{-1}\otimes K$.
For (i), if $\deg (D)=2g-2-2\deg \ell$, then $\deg L=\deg \ell$.  However, since $\ME$ has rank $2$ and $\deg\;\ME=0$,
there is a unique subbundle of positive degree, so $L= \ell$. By the form of the Higgs bundle, we conclude that
$D=\al$. For (ii), if $2g-2>\deg D>2g-2-2\deg \ell$, if there is solution with line bundle $L$, then $0<\deg L<\deg \ell$,
which is impossible.
\end{proof}

\bibliographystyle{plain}
\bibliography{references}

\begin{thebibliography}{10}

\bibitem{corlette1988}
Kevin Corlette.
\newblock Flat {$G$}-bundles with canonical metrics.
\newblock {\em J. Differential Geom.}, 28(3):361--382, 1988.

\bibitem{donaldson1985anti}
Simon~K. Donaldson.
\newblock Anti-self-dual {Y}ang-{M}ills connections over complex algebraic
  surfaces and stable vector bundles.
\newblock {\em Proc. London Math. Soc. (3)}, 50(1):1--26, 1985.

\bibitem{gaiotto2012knot}
Davide Gaiotto and Edward Witten.
\newblock Knot invariants from four-dimensional gauge theory.
\newblock {\em Advances in Theoretical and Mathematical Physics},
  16(3):935--1086, 2012.

\bibitem{He2017}
Siqi He.
\newblock A gluing theorem for the {K}apustin-{W}itten equations with a {N}ahm
  pole.
\newblock {\em arXiv preprint arXiv:1707.06182}, 2017.

\bibitem{Heexpansion18}
Siqi He.
\newblock The expansions of the {N}ahm pole solutions to the
  {K}apustin-{W}itten equations.
\newblock {\em arXiv preprint arXiv:1808.03886}, 2018.

\bibitem{HeMazzeo2017}
Siqi He and Rafe Mazzeo.
\newblock The extended {B}ogomolny equations and generalized {N}ahm pole
  boundary conditions.
\newblock {\em arXiv preprint arXiv:1710.10645}, 2017.

\bibitem{MazzeoHe18}
Siqi He and Rafe Mazzeo.
\newblock The extended {B}ogomolny equations with generalized {N}ahm pole
  boundary conditions, {II}.
\newblock {\em arXiv preprint arXiv:1806.06314}, 2018.

\bibitem{Hitchin1987Selfdual}
Nigel Hitchin.
\newblock The self-duality equations on a {R}iemann surface.
\newblock {\em Proc. London Math. Soc. (3)}, 55(1):59--126, 1987.

\bibitem{hitchin1987stable}
Nigel Hitchin.
\newblock Stable bundles and integrable systems.
\newblock {\em Duke mathematical journal}, 54(1):91--114, 1987.

\bibitem{hitchin1992lie}
Nigel Hitchin.
\newblock Lie groups and {T}eichm\"uller space.
\newblock {\em Topology}, 31(3):449--473, 1992.

\bibitem{KapustinWitten2006}
Anton Kapustin and Edward Witten.
\newblock Electric-magnetic duality and the geometric {L}anglands program.
\newblock {\em Commun. Number Theory Phys.}, 1(1):1--236, 2007.

\bibitem{MazzeoWitten2013}
Rafe Mazzeo and Edward Witten.
\newblock The {N}ahm pole boundary condition.
\newblock {\em The influence of Solomon Lefschetz in geometry and topology.
  Contemporary Mathematics}, 621:171--226, 2013.

\bibitem{MazzeoWitten2017}
Rafe Mazzeo and Edward Witten.
\newblock The {KW} equations and the {N}ahm pole boundary condition with knot.
\newblock {\em arXiv preprint arXiv:1712.00835}, 2017.

\bibitem{Mikhaylov2012solutions}
Victor Mikhaylov.
\newblock On the solutions of generalized {B}ogomolny equations.
\newblock {\em Journal of High Energy Physics}, 2012(5):112, 2012.

\bibitem{RyosukeEnergy}
Ryosuke Takahashi and Naichung Leung.
\newblock Energy bound for {K}apustin-{W}itten solutions on
  {$S^3\times\mathbb{R}^+$}.
\newblock {\em arXiv preprint arXiv:1801.04412}, 2018.

\bibitem{taubes1982self}
Clifford Taubes.
\newblock Self-dual {Y}ang-{M}ills connections on non-self-dual 4-manifolds.
\newblock {\em Journal of Differential Geometry}, 17(1):139--170, 1982.

\bibitem{taubes2013compactness}
Clifford Taubes.
\newblock Compactness theorems for ${SL}(2;\mathbb{C})$ generalizations of the
  4-dimensional anti-self dual equations.
\newblock {\em arXiv preprint arXiv:1307.6447}, 2013.

\bibitem{Taubescompactness}
Clifford Taubes.
\newblock Sequences of {N}ahm pole solutions to the {S}{U}(2)
  {K}apustin-{W}itten equations.
\newblock {\em arXiv preprint arXiv:1805.02773}, 2018.

\bibitem{uhlenbeck1986existence}
Karen Uhlenbeck and S.-T. Yau.
\newblock On the existence of {H}ermitian-{Y}ang-{M}ills connections in stable
  vector bundles.
\newblock {\em Comm. Pure Appl. Math.}, 39(S, suppl.):S257--S293, 1986.
\newblock Frontiers of the mathematical sciences: 1985 (New York, 1985).

\bibitem{witten2011fivebranes}
Edward Witten.
\newblock Fivebranes and knots.
\newblock {\em Quantum Topol.}, 3(1):1--137, 2012.

\bibitem{Witten2014LecturesJonesPolynomial}
Edward Witten.
\newblock Two lectures on the {J}ones polynomial and {K}hovanov homology.
\newblock In {\em Lectures on geometry}, Clay Lect. Notes, pages 1--27. Oxford
  Univ. Press, Oxford, 2017.

\end{thebibliography}
\end{document}